\documentclass{amsart}
\usepackage[utf8]{inputenc}
\usepackage{color}
\usepackage{amssymb}
\usepackage{amsthm}
\usepackage{amsmath}
\usepackage{mathtools}
\usepackage{graphicx}
\usepackage{hyperref}
\usepackage{comment}

\newtheorem{theorem}{Theorem}[section]
\newtheorem{corollary}[theorem]{Corollary}
\newtheorem{lemma}[theorem]{Lemma}
\newtheorem{definition}[theorem]{Definition}

\title{Topological symmetry groups of the Petersen graph}
\author{D.~Chambers}
\author{E. Flapan}
\author{D.~Heath}
\author{E.~Lawrence}
\author{C.~Thatcher}
\author{R.~Vanderpool}
\address{Teaching and Learning Center, University of Washington Tacoma, Tacoma, WA 98402, USA}
\address{Department of Mathematics, Pomona College, Claremont, CA 91711}
\address{Department of Mathematics, Pacific Lutheran University, Tacoma, WA 98447}
\address{Department of Mathematics and Statistics, University of San Francisco, San Francisco, CA 94117, USA}
\address{Department of Mathematics, University of Puget Sound, Tacoma, WA 98416}
\address{School of Interdisciplinary Arts and Sciences, University of Washington Tacoma,
Tacoma, WA 98402, USA}
\date{\today}


\thanks{The second author was supported in part by NSF Grant DMS-1607744}

\begin{document}

\maketitle

\begin{abstract}We characterize all groups which can occur as the topological symmetry group or the orientation preserving topological symmetry group of some embedding of the Petersen graph in $S^3$.  \end{abstract}

\section{Introduction}\label{introduction}

The symmetries of a molecule can predict many of its chemical properties, but non-rigid molecules may have symmetries which are not immediately apparent.  The field of spatial graph theory developed largely in order to study the symmetries of such molecular structures \cite{Si}.   Since small molecules are normally rigid, their symmetries are described by their rotations and reflections in space, which together form the group of rigid symmetries known as the \emph{point group} of a molecule. However, due to their flexibility, large molecules may have symmetries that are not contained in their point group. In this case, we represent the symmetries of the molecule by the subgroup of the automorphism group of the molecular graph that can be induced by homeomorphisms of the graph in $3$-dimensional space. This group considers a molecule as a topological object in space, and hence is called the \emph{topological symmetry group} of the structure. 

The topological symmetry group can be used to represent the symmetries of any spatial graph whether or not it is a molecular graph.  While molecular graphs are normally considered in $\mathbb{R}^3$, an automorphism of a spatial graph is induced by a homeomorphism of $\mathbb{R}^3$ if and only if it is induced by a homeomorphism of $S^3$.  Thus the topological symmetry group of a spatial graph is the same whether the graph is considered in $S^3$ or in $\mathbb{R}^3$. However, symmetries are easier to visualize in $S^3$ than in $\mathbb{R}^3$ because of the existence of glide rotations.  Thus topological symmetry groups are normally defined for graphs embedded in $S^3$.

We begin with some definitions.

 \begin{definition}  Let $\gamma$ be an abstract graph.  The group of automorphisms of the vertices of $\gamma$ is denoted by $\mathrm{Aut}(\gamma)$.
 \end{definition}

\begin{definition}  Let $\Gamma$ be a graph embedded in $S^3$.  We define the {\bf topological symmetry group} $\mathrm{TSG}(\Gamma)$ as the subgroup of $\mathrm{Aut}(\Gamma)$ induced by homeomorphisms of $(S^3,\Gamma)$.  If we only consider orientation preserving homeomorphisms we obtain the {\bf orientation preserving topological symmetry group} $\mathrm{TSG}_+(\Gamma)$.
\end{definition}

\begin{definition}  Let $\gamma$ denote an abstract graph and let $G$ be a subgroup of $\mathrm{Aut}(\gamma)$.  If there is an embedding $\Gamma$ of $\gamma$ in $S^3$ such that $\mathrm{TSG}(\Gamma)=G$, then we say that the group $G$ is {\bf realizable} for $\gamma$ and that $\Gamma$ {\bf realizes} $G$.
If there is some embedding $\Gamma$ of $\gamma$ in $S^3$ such that $\mathrm{TSG}_+(\Gamma)= G$, then we say that the group $G$ is {\bf positively realizable} for $\gamma$ and that $\Gamma$ {\bf positively realizes} $G$.
\end{definition}

For an abstract graph $\gamma$, we make the following observation about the relationship between groups that are realizable for $\gamma$ and groups that are positively realizable for $\gamma$. Suppose that $G$ is positively realized by some embedding $\Gamma$ of $\gamma$.  By adding identical trefoil knots to every edge of $\Gamma$ we obtain an embedding $\Lambda$ which is not invariant under any orientation reversing homeomorphism of $S^3$, but such that every automorphism that was induced on $\Gamma$ by an orientation preserving homeomorphism of $S^3$ is also induced on $\Lambda$ by an orientation preserving homeomorphism of $S^3$.  Thus $\mathrm{TSG}(\Lambda)=\mathrm{TSG}_+(\Gamma)$.  It follows that every group which is positively realizable for $\gamma$ is also realizable for $\gamma$.

Topological symmetry groups have been studied for graphs embedded in $3$-manifolds as well as in $S^3$.  In particular, it was shown in \cite{TSG1} that no alternating group $A_n$ with $n>5$, can occur as the topological symmetry group of a graph embedded in $S^3$.  By contrast, it was shown in \cite{TSG3} that every finite group is the topological symmetry group of a graph embedded in a hyperbolic rational homology sphere; though for any given closed, connected, orientable, irreducible, $3$-manifold, there is some $n$ such that the alternating group $A_n$ is not the topological symmetry group of any graph embedded in that manifold.  Furthermore, it was shown in \cite{CK} that the orientation preserving topological symmetry group of a $3$-connected spatial graph in $S^3$ is isomorphic to its orientation preserving mapping class group if and only if the complement of the graph is atoroidal. 

Topological symmetry groups have been completely characterized for several particular families of spatial graphs in $S^3$.  For the complete graphs, all possible orientation preserving topological symmetry groups were characterized in a series of papers \cite{CF1,CF2, FMN, FMN2, FMNY, TSG2}.  For the family of M\"{o}bius ladders with $n$ rungs (which includes the bipartite graph $K_{3,3}$), all orientation preserving topological symmetry groups were characterized in \cite{FL}.  For the family of complete bipartite graphs $K_{n,n}$, it was shown in \cite{TSG1} that all finite subgroups of $\mathrm{SO}(4)$ can be positively realized by an embedding of some $K_{n,n}$ in $S^3$.  More recently, \cite{Me} classified all $n$ such that the polyhedral groups $A_4$, $S_4$, or $A_5$ can be positively realized for $K_{n,n}$; and \cite{Me2} classified all $n$ such that the groups $\mathbb{Z}_m$, $D_m$, $\mathbb{Z}_r \times \mathbb{Z}_s$ or $(\mathbb{Z}_r \times \mathbb{Z}_s) \ltimes \mathbb{Z}_2$ can be positively realized for $K_{n,n}$.

In this paper, we determine all groups which can occur as the topological symmetry group or the orientation preserving topological symmetry group of some embedding of the Petersen graph in $S^3$.  We focus on the Petersen graph because of its importance in providing counterexamples to conjectures in graph theory \cite{DHJS} and because of the role it has played in the study of intrinsically linked graphs (see for example, \cite{NT, rst, S1, S2}).  As an abstract graph, the Petersen graph (henceforth denoted by $P$) is the $3$-connected graph with $10$ vertices and $15$ edges illustrated in Figure~\ref{pairs}.

\begin{figure}[h!]
\includegraphics[width=6cm]{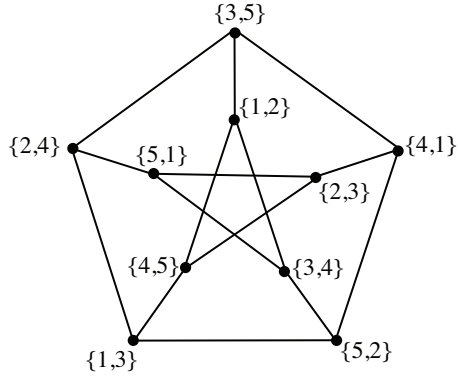}
\caption{The Petersen graph with vertices labeled by pairs of numbers such that there is an edge between two vertices if and only if their pairs of numbers are disjoint.}\label{pairs}
\end{figure}

In Figure~\ref{pairs}, we have labeled the vertices of $P$ with pairs of numbers between $1$ and $5$ such that there is an edge between two vertices if and only if the corresponding pairs of numbers are disjoint.  An automorphism of $P$ can then be represented by an element of the symmetric group $S_5$ describing its action on each number within the pairs labeling the vertices.  For example, the automorphism obtained by rotating Figure~\ref{pairs} by $-\frac{2\pi}{5}$ can be represented by the permutation $(12345)$.  In fact, it is shown in \cite{DHJS} that the automorphism group of the Petersen graph is isomorphic to $S_5$.  Thus we are interested in which subgroups of $S_5$ are realizable or positively realizable for $P$.

The trivial group can be realized and positively realized by an embedding of $P$ where each edge contains a distinct knot.  The following is a complete list of the non-trivial subgroups of $S_5$, up to isomorphism \cite{P}, where we use $D_n$ to denote the dihedral group of order $2n$.

$$S_5,\  S_4,\  A_5,\  A_4,\  \mathbb{Z}_5 \rtimes \mathbb{Z}_4,\  D_6,\ D_5, \  D_4,\ D_3,\  D_2,\ \mathbb{Z}_6,\  \mathbb{Z}_5,\  \mathbb{Z}_4,\ \mathbb{Z}_3,\ \mathbb{Z}_2$$

Our main results are the following.

\begingroup
\def\thetheorem{\ref{positive}}
\begin{theorem}
The non-trivial groups which are positively realizable for the Petersen graph are $D_5$, $D_3$, $\mathbb{Z}_5$, $\mathbb{Z}_3$, and $\mathbb{Z}_2$.
\end{theorem}
\addtocounter{theorem}{-1}
\endgroup

\begingroup
\def\thetheorem{\ref{Realize}}
\begin{theorem}
The non-trivial groups which are realizable for the Petersen graph are $\mathbb{Z}_5 \rtimes \mathbb{Z}_4$, $D_5$, $D_3$, $\mathbb{Z}_5$, $\mathbb{Z}_4$, $\mathbb{Z}_3$, and $\mathbb{Z}_2$.
\end{theorem}
\addtocounter{theorem}{-1}
\endgroup


In Section~2, we prove that the groups $S_5$, $S_4$, $A_5$, $A_4$, $\mathbb{Z}_5 \rtimes \mathbb{Z}_4$, $D_6$, $D_4$, $D_2$, $\mathbb{Z}_6$, and $\mathbb{Z}_4$ are not positively realizable for the Petersen graph, and that all of these groups except $\mathbb{Z}_5 \rtimes \mathbb{Z}_4$ and $\mathbb{Z}_4$ are also not realizable for the Petersen graph.  Then in Section~3, we construct embeddings of the Petersen graph which positively realize or realize the remaining subgroups of $S_5$. 

\bigskip

\section{Negative results}\label{negResult}

Observe that no cycles in $P$ have fewer than $5$ vertices and the only pairs of disjoint cycles in the Petersen graph are $5$-cycles. We begin with a lemma about pairs of disjoint $5$-cycles in $P$.  

\begin{lemma}\label{abstractP}  The Petersen graph contains six pairs of disjoint $5$-cycles, and no non-trivial automorphism fixes every vertex of a $5$-cycle.
\end{lemma}

\begin{proof} Since each vertex is used in every pair of disjoint $5$-cycles, we obtain every such pair by listing all $5$-cycles containing the vertex $\{1,2\}$. We see from the list below that there are exactly six such $5$-cycles.

\begin{align*}\{1,2\},\{3,4\},\{1,5\},\{2,4\},\{3,5\}\\
\{1,2\},\{3,4\},\{1,5\},\{2,3\},\{4,5\}\\
\{1,2\},\{3,4\},\{2,5\},\{1,3\},\{4,5\}\\
\{1,2\},\{3,4\},\{2,5\},\{1,4\},\{3,5\}\\
\{1,2\},\{3,5\},\{1,4\},\{2,3\},\{4,5\}\\
\{1,2\},\{3,5\},\{2,4\},\{1,3\},\{4,5\}\\\end{align*}

Now suppose that $\varphi$ is an automorphism of $P$ which fixes every vertex on a $5$-cycle $X$.  Any vertex  in $X$ has two incident edges in $X$ and one incident edge not in $X$.  Thus $\varphi$ has to fix the one incident edge $e$ not in $X$.  This implies that $\varphi$ fixes both vertices of $e$.  Since $P$ has no cycles with fewer than $5$ vertices, only one vertex of $e$ can be on $X$.  Hence $\varphi$ fixes a vertex of $e$ that is on the complementary $5$-cycle $Y$.  Since $\varphi$ fixes every vertex of $X$, it fixes all of the edges joining $X$ and $Y$, and hence every vertex of $Y$.  Since all $10$ vertices of $P$ are on $X\cup Y$, this implies that $\varphi$ is the identity.  \end{proof}

\begin{theorem}\label{noD_2} Let $G$ be a group containing $D_2$.  Then $G$  is not realizable or positively realizable for the Petersen graph.
\end{theorem}

\begin{proof}   Suppose that $\Gamma$ realizes $G$ for the Petersen graph.  Since $D_2\leq G$,  there are homeomorphisms $g$ and $h$ of $(S^3, \Gamma)$ inducing automorphisms $\varphi$ and $\psi$ respectively  on $\Gamma$ such that $\langle \varphi, \psi\rangle= D_2$, where $\varphi^2=\psi^2=1$ and $\varphi\psi=\psi\varphi$.

Let $C_1,\dots, C_6$ denote the six pairs of disjoint $5$-cycles in $\Gamma$.  It follows from \cite{NT} that $\mathrm{lk}(C_1)+...+\mathrm{lk}(C_6)\equiv 1\ ({\rm mod}\ 2)$.  Thus an odd number of the $C_i$ have odd linking number and an odd number of the $C_i$ have even linking number.  Also, regardless of whether $g$ and $h$ are orientation preserving or reversing, they each preserve the parity of the linking number.  Thus $\varphi$ and $\psi$ each permute the elements within the set $O$ of all $C_i$ with odd linking number and within the set $E$ of all $C_i$ with even linking number.  Since $\varphi$ and $\psi$ are involutions and $|O|$ and $|E|$ are both odd, they each setwise fix at least one $C_i$ in $O$ and at least one $C_i$  in $E$.  

Now suppose for the sake of contradiction that no $C_i$ is setwise fixed by both $\varphi$ and $\psi$.  Then $O$ and $E$ must each contain three elements.  Thus without loss of generality $O=\{C_1,C_2,C_3\}$, $\psi(C_1)=C_1$, $\varphi(C_2)=C_2$, and $\varphi(C_1)=C_3$.  Then $\psi(C_3)=\psi\varphi(C_1)=\varphi\psi(C_1)=\varphi(C_1)=C_3$.  Thus $\psi(C_3)=C_3$.  It follows that $\psi$ setwise fixes each of $C_1$, $C_2$, and $C_3$.  Hence there is a $C_i$ which is setwise fixed by both $\varphi$ and $\psi$.   Without loss of generality, $C_1$ is fixed by both $\psi$ and $\varphi$. 

 Let $X_1$ and $Y_1$ be the $5$-cycles in $C_1$, and suppose that $\varphi(X_1)=X_1$ and $\varphi(Y_1)=Y_1$.  By Lemma~\ref{abstractP}, $\varphi$ cannot fix every vertex on $X_1$ or on $Y_1$.  Since $\mathrm{order}(\varphi)=2$ and $X_1$ and $Y_1$ are $5$-cycles, this implies that there are unique vertices $v$ on $X_1$ and $w$ on $Y_1$ which are fixed by $\varphi$.  Now $\varphi\psi(v)=\psi\varphi(v)=\psi(v)$.  Since $v$ is the unique vertex on $X_1$ which is fixed by $\varphi$, this implies that $\psi(v)=v$.  Similarly, $\psi(w)=w$.  But since $\psi$ setwise fixes $C_1=X_1\cup Y_1$, it follows that $\psi(X_1)=X_1$ and $\psi(Y_1)=Y_1$.  Now $\varphi$ and $\psi$ are both involutions of $P$ leaving $X_1$ and $Y_1$ setwise invariant fixing $v$ and $w$.  But this implies that $\varphi=\psi$, which is impossible since $\langle \varphi, \psi\rangle=D_2$.  Thus $\varphi$ cannot setwise fix $X_1$ and $Y_1$.  Similarly, $\psi$ cannot setwise fix $X_1$ and $Y_1$. 

Thus both $\varphi$ and $\psi$ interchange $X_1$ and $Y_1$.  Now let $\alpha=\varphi\psi$.  Then $\alpha(X_1)=X_1$ and $\alpha (Y_1)=Y_1$, and $\alpha$ has order $2$ and commutes with both $\varphi$ and $\psi$.  But now we can replace $\varphi$ by $\alpha$ in the above argument to again get a contradiction.  It follows that $G$ is not realizable for the Petersen graph, and hence $G$ is also not positively realizable for the Petersen graph.\end{proof}

Observe that since the groups $S_5$, $S_4$, $A_5$ $A_4$, $D_6$, $D_4$, and $D_2$ all contain $D_2$, the following corollary is immediate.

\begin{corollary}  \label{cor}The groups $S_5$, $S_4$, $A_5$ $A_4$, $D_6$, $D_4$, and $D_2$  are not realizable or positively realizable for the Petersen graph.
\end{corollary}

\begin{theorem}\label{T0}  Let $\Gamma$ be an embedding of the Petersen graph in $S^3$ and let $g$ be a homeomorphism of $(S^3,\Gamma)$. Then $g$ cannot induce an automorphism of $\Gamma$ of order $6$; and if $g$ is orientation preserving then $g$ also cannot induce an automorphism of $\Gamma$ of order $4$.  \end{theorem}

\begin{proof}  Suppose that $g$ induces an automorphism $\varphi$ of $\Gamma$ of order $4$ or $6$. Since the Petersen graph is $3$-connected, we can apply Theorem 1 of \cite{EF} to obtain an embedding $\Gamma'$ of $P$ in $S^3$ such that $\varphi$ is induced on $\Gamma'$ by a finite order homeomorphism $h:(S^3,\Gamma')\to (S^3,\Gamma')$. Furthermore, since $P$ is non-planar, by Smith Theory \cite{Smith}, a finite order homeomorphism of $S^3$ cannot pointwise fix $P$. Hence the order of $h$ must be the same as the order of $\varphi$.

 If $\varphi$ has order $4$, then without loss of generality, we can label the vertices of $\Gamma'$ with $2$-element sets of numbers so that $\varphi$ is the element $(1234)$ of $S_5$; and if $\varphi$ has order $6$, then without loss of generality, we can label the vertices of $\Gamma'$ with $2$-element sets of numbers so that $\varphi$ is the element $(123)(45)$ of $S_5$.  If $\varphi=(1234)$, then some point on the edge with vertices $\{2,4\}$ and $\{1,3\}$ is fixed by $h$; and if $\varphi=(123)(45)$, then vertex $\{4,5\}$ is fixed by $h$. Let $\mathrm{fix}(h)$ denote the fixed point set of $h$, then in either case, $\mathrm{fix}(h)\not=\emptyset$.  Hence by Smith Theory \cite{Smith}, if $h$ is orientation preserving then $\mathrm{fix}(h)$ is homeomorphic to $S^1$, and if $h$ is orientation reversing then $\mathrm{fix}(h)$ consists of two points or is homeomorphic to $S^2$.  
 
Now suppose $h$ has order $6$.  Then $h$ induces $\varphi=(123)(45)$ on $P$, and hence $h^3$ induces $(45)$ on $P$.  Thus the vertices $\{1,2\}$, $\{2,3\}$, and $\{1,3\}$ are fixed by $h^3$ but were not fixed by $h$.  This means that $\mathrm{fix}(h^3)$ is larger than $\mathrm{fix}(h)$.  Thus $h$ must be orientation reversing, and $h^3$ must pointwise fix a $2$-sphere containing vertices $\{4,5\}$, $\{1,2\}$, $\{2,3\}$, and $\{1,3\}$ and no other vertices.  Let $A$ and $B$ be the components of $S^3-\mathrm{fix}(h^3)$.  Then $h^3$ interchanges $A$ and $B$.  

Since $h^3$ induces the automorphism $(45)$ on $P$, without loss of generality $\{1,4\}\in A$ and $\{1,5\}\in B$.  The Petersen graph has an edge with vertices $\{1,4\}$ and $\{2,5\}$ which is not setwise fixed by $h^3$.  Hence this edge must be disjoint from $\mathrm{fix}(h^3)$.  It follows that $\{2,5\}\in A$.  However, there are also edges between $\{3,4\}$ and both $\{1,5\}\in B$ and $\{2,5\}\in A$, and again neither of these edges are setwise fixed by $h^3$.  But this would imply that $\{3,4\}\in B$ and $\{3,4\}\in A$, which is impossible.  Thus $h$ cannot have order $6$.

Finally, suppose that $h$  is orientation preserving and has order $4$.  Then since $\mathrm{fix}(h)\not=\emptyset$, $\mathrm{fix}(h)$ and $\mathrm{fix}(h^2)$ are both homeomorphic to $S^1$.  Since every point which is fixed by $h$ must also be fixed by $h^2$, in fact $\mathrm{fix}(h)= \mathrm{fix}(h^2)$.  However, since $h$ induces $(1234)$ on $\Gamma$, vertices $\{2,4\}$ and $\{1,3\}$ are fixed by $h^2$ but were not fixed by $h$. By this contradiction we conclude that if $h$ is orientation preserving, then $h$ cannot have order $4$.  \end{proof}

The following corollary is immediate from Theorem~\ref{T0}.

\begin{corollary}\label{46cor} The groups $\mathbb{Z}_5\rtimes\mathbb{Z}_4,\ \mathbb Z_6$, and $\mathbb Z_4$ are not positively realizable for the Petersen graph, and the group $\mathbb{Z}_6$ is not realizable for the Petersen graph.
\end{corollary}

\bigskip

\section{Positive results}\label{mainResult}

Now we prove Theorem~\ref{positive}, which we restate here for the convenience of the reader.

\begin{theorem}\label{positive}
The non-trivial groups which are positively realizable for the Petersen graph are $D_5$, $D_3$, $\mathbb{Z}_5$, $\mathbb{Z}_3$, and $\mathbb{Z}_2$.
\end{theorem}

\begin{proof}  Corollaries~\ref{cor} and \ref{46cor} show that the groups $S_5$, $S_4$, $A_5$, $A_4$, $\mathbb{Z}_5 \rtimes \mathbb{Z}_4$, $D_6$, $D_4$, $D_2$, $\mathbb{Z}_6$,  and $\mathbb{Z}_4$ are not positively realizable for the Petersen graph.  Thus the only groups that could be positively realizable are $D_5$, $D_3$, $\mathbb{Z}_5$, $\mathbb{Z}_3$, and $\mathbb{Z}_2$. Below we provide embeddings of the Petersen graph which positively realize each of these groups.

First let $\Gamma$ be the embedding shown in Figure~\ref{D5Realize}, with vertex labels as indicated. The $5$-cycle $abcde$ must be setwise invariant under every homeomorphism of $(S^3,\Gamma)$ because it is the only $5_1$ knot in the embedding.  Thus $\mathrm{TSG}_+(\Gamma)\leq D_5$.  The homeomorphism of $S^3$ which rotates the figure by $-\frac{2\pi}{5}$ induces the automorphism $\varphi=(12345)(acebd)$ on $\Gamma$; and the homeomorphism of $S^3$ which turns the figure over induces the automorphism $\psi=(25)(34)(eb)(cd)$ on $\Gamma$.  Since $\langle \varphi, \psi\rangle=D_5$, we have $\mathrm{TSG}_+(\Gamma)= D_5$.

\begin{figure}[h!]
\includegraphics[width=4cm]{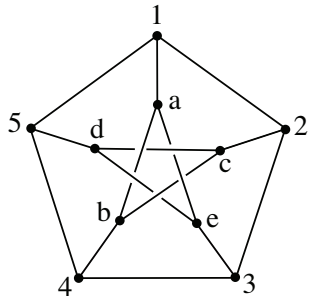}
\caption{An embedding $\Gamma$ of $P$ with $\mathrm{TSG}_+(\Gamma)=D_5$}\label{D5Realize}
\end{figure}

We obtain an embedding $\Gamma_1$ from $\Gamma$ by adding identically oriented $8_{17}$ knots to each of the edges of the cycle $12345$.  Since $abcde$ is still the only $5_1$ knot in $\Gamma_1$, $abcde$ must still be setwise invariant, and hence again $\mathrm{TSG}_+(\Gamma_1)\leq D_5$.  Also, the automorphism $\varphi=(12345)(acebd)$ is still induced by a rotation of $(S^3,\Gamma_1)$ by $-\frac{2\pi}{5}$.  But by the non-invertibility of  the knot $8_{17}$, no homeomorphism can turn over the $5$-cycle $abcde$.  Thus $\mathrm{TSG}_+(\Gamma)= \mathbb{Z}_5$.

Next we let $\Gamma_2$ be the embedding obtained from $\Gamma$ by adding the knot $4_1$ to the edges $15$ and $12$.  Again  the $5$-cycle $abcde$ is setwise invariant, and hence $\mathrm{TSG}_+(\Gamma_2)\leq D_5$.  Since the $4_1$ knot is invertible, turning the figure over still induces the automorphism $\psi=(25)(34)(eb)(cd)$.  However, now no rotation of $abcde$ is possible.  Thus we have $\mathrm{TSG}_+(\Gamma_2)= \mathbb{Z}_2$.

We now start with the embedding $\Lambda$ of the Petersen graph shown in Figure \ref{D3em}. Observe that the $9$-cycle $154baedc2$ contains a trefoil knot, which is the only knot in the embedding.  Also vertex $3$ is the only vertex which is not on or adjacent to this knotted $9$-cycle.  So any homeomorphism of $(S^3,\Lambda)$ must setwise fix the set of edges $\{23, e3, 43\}$.  Thus $\mathrm{TSG}_+(\Lambda)\leq D_3$. Now the homeomorphism of $S^3$ which rotates the figure by $-\frac{2\pi}{3}$ about an axis perpendicular to the page going through vertex $3$ induces the automorphism $\varphi=(24e)(c5a)(bd1)$ on $\Lambda$.  Also, the homeomorphism of $S^3$ which turns the figure over induces the automorphism $\psi=(1c)(ab)(e4)(d5)$.  Since $\langle\varphi,\psi\rangle=D_3$, we have $\mathrm{TSG}_+(\Lambda)= D_3$.

\begin{figure}[ht]
\includegraphics[width=5cm]{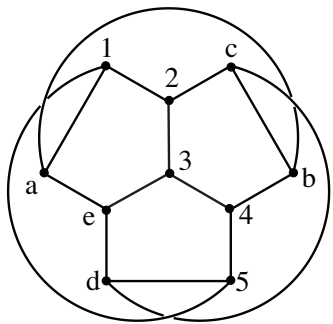}
\caption{An embedding $\Lambda$ of $P$ with $\mathrm{TSG}_+(\Lambda)=D_3$.}\label{D3em}
\end{figure}

Finally, we obtain an embedding $\Lambda_1$ from $\Lambda$ by adding identically oriented $8_{17}$ knots to the edges $a1$, $cb$, and $5d$.  The set of edges $\{23, e3, 43\}$ must again be invariant under any homeomorphism of $(S^3,\Lambda_1)$, and hence $\mathrm{TSG}_+(\Lambda_1)\leq D_3$.  Also, the automorphism $(24e)(c5a)(bd1)$ is still induced by a rotation of $\Lambda_1$ by $-\frac{2\pi}{3}$.  However, because of the non-invertibility of $8_{17}$, no homeomorphism can turn over the knotted $9$-cycle $154baedc2$.  Hence $\mathrm{TSG}_+(\Lambda_1)= \mathbb{Z}_3$.  \end{proof}

Finally, we prove Theorem~\ref{Realize}, which we restate here for the convenience of the reader.

\begin{theorem}\label{Realize}
The non-trivial groups which are realizable for the Petersen graph are $\mathbb{Z}_5 \rtimes \mathbb{Z}_4$, $D_5$, $D_3$, $\mathbb{Z}_5$, $\mathbb{Z}_4$, $\mathbb{Z}_3$, and $\mathbb{Z}_2$.
\end{theorem}

\begin{proof} Since any positively realizable group is also realizable, the positively realizable groups $D_5$, $D_3$, $\mathbb{Z}_5$, $\mathbb{Z}_3$, and $\mathbb{Z}_2$ are realizable for the Petersen graph.  Furthermore, by Corollaries~\ref{46cor} and \ref{cor}, the groups $S_5$, $S_4$, $A_5$, $A_4$, $D_6$, $D_4$, $D_2$, and $\mathbb{Z}_6$ are not realizable. Hence we are left to determine the realizability of the groups $\mathbb{Z}_5 \rtimes \mathbb{Z}_4$ and  $\mathbb{Z}_4$.

Let $\Delta$ denote the embedding of the Petersen graph in $S^3$ illustrated on the left side of Figure~\ref{Z5Z4}, where $X$ denotes the horizontal circle whose vertices are labeled with numbers and $Y$ denotes the vertical circle whose vertices are labeled with letters.  Let $g$ be the glide rotation which rotates $X$ by $\frac{2\pi}{5}$ and rotates $Y$ by $\frac{4\pi}{5}$, in each case in the direction indicated by the arrow.  Then $g$ induces the automorphism $\varphi=(12345)(acebd)$ on $\Delta$, as shown on the right side of the figure.

\begin{figure}[h!]
\includegraphics[width=12cm]{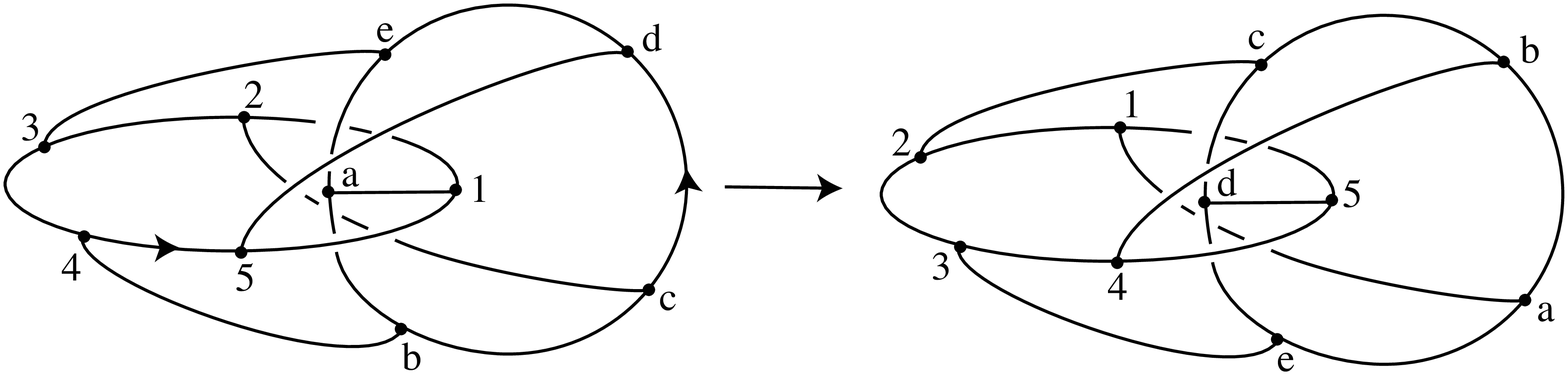}
\caption{A glide rotation induces $\varphi=(12345)(acebd)$.}
\label{Z5Z4}
\end{figure}

Now let $h$ denote the orientation reversing homeomorphism which first rotates $\Delta$ by $\frac{\pi}{2}$ about a horizontal axis so that $X$ becomes vertical and $Y$ becomes horizontal, and then reflects $\Delta$ through a vertical mirror, as illustrated in Figure~\ref{Z5Z4Reflect}.  Then $h$ induces $\psi=(1a)(2b5e)(3c4d)$ on $\Delta$.

\begin{figure}[h!]
\includegraphics[width=12cm]{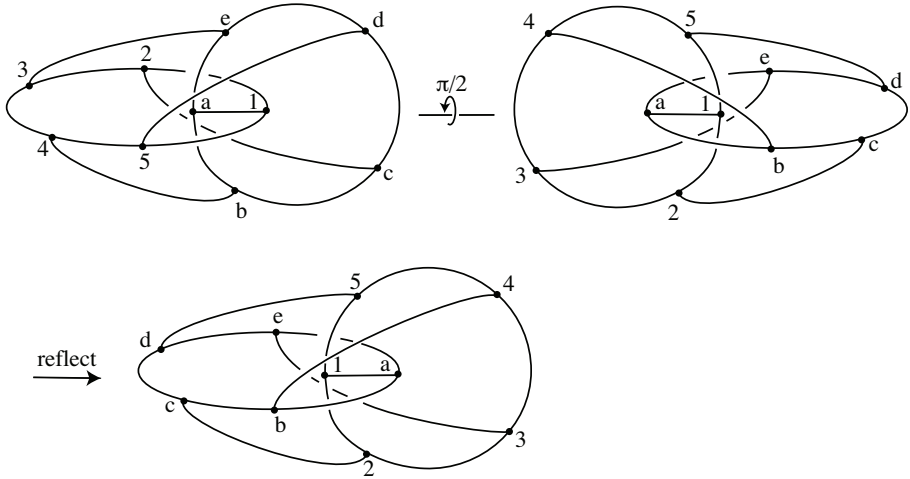}
\caption{A rotation followed by a reflection induces $\psi=(1a)(2b5e)(3c4d)$.}
\label{Z5Z4Reflect}
\end{figure}

Observe that $\varphi^5$ and $\psi^4$ are each the identity, and $\varphi\psi=(1c5b)(2d4a)(3e)=\psi\varphi^2$.  Thus $\langle\psi,\varphi\rangle=\mathbb{Z}_5 \rtimes \mathbb{Z}_4\leq \mathrm{TSG}(\Delta)$.  Furthermore, since no group containing $\mathbb{Z}_5 \rtimes \mathbb{Z}_4$ as a proper subgroup is among those that could be $\mathrm{TSG}(\Delta)$, we must have $\mathrm{TSG}(\Delta)=\mathbb{Z}_5 \rtimes \mathbb{Z}_4$ as required.

Finally, we obtain an embedding $\Delta_1$ from the embedding $\Delta$ illustrated on the left in Figure~\ref{Z5Z4} by adding the knot $4_1$ to the edge $1a$, the knot $+3_1$ to each of the edges of the $5$-cycle $abcde$, and the knot $-3_1$ to each of the edges of the $5$-cycle $12345$.  Since $4_1$ is invertible and achiral, $\psi=(1a)(2b5e)(3c4d)$ is still induced on $\Delta_1$ by a rotation followed by a reflection which interchanges the $+3_1$ knots on the edges of $abcde$ with the $-3_1$ knots on the edges of $12345$.  Thus $\mathbb{Z}_4\leq \mathrm{TSG}(\Delta_1)$, and from our list of realizable groups we know that either $\mathrm{TSG}(\Delta_1)=\mathbb{Z}_4$ or $\mathrm{TSG}(\Delta)=\mathbb{Z}_5 \rtimes \mathbb{Z}_4$.  

Suppose that some homeomorphism $g$ of $S^3$ induces an order $5$ automorphism on $\Delta_1$.  Because $1a$ is the only edge of $\Delta_1$ containing a $4_1$ knot, $g$ must setwise fix $1a$.  Now $g^2$ is orientation preserving and still induces an order $5$ automorphism on $\Delta_1$.  Since $12345$ contains five $+3_1$ knots and $abcde$ contains five $-3_1$ knots, this implies that $g^2$ takes the $5$-cycles $12345$ and $abcde$ to themselves fixing vertices $1$ and $a$.   But this is impossible since $g^2$ has order $5$.  Hence $\Delta_1$ cannot be invariant under a homeomorphism inducing an order $5$ automorphism on $\Delta_1$.  It now follows that $\mathrm{TSG}(\Delta_1)=\mathbb{Z}_4$, as required.
 \end{proof}

\bigskip


\end{document}